\documentclass[12pt,reqno,psamsfonts]{amsart}

\usepackage{graphicx}
\usepackage{amsmath}
\usepackage{amssymb}
\usepackage{enumerate}
\usepackage{verbatim}
\usepackage{url}
\usepackage{tikz-cd}

\usepackage{dsfont}

\usepackage{hyperref} 
\hypersetup{colorlinks=true}
\usepackage{hypcap}

\numberwithin{equation}{section}
\newcommand{\BB}{\mathds{1}}

\newcommand{\R}{\mathbb{R}}
\newcommand{\C}{\mathbb{C}}
\newcommand{\T}{\mathbb{T}}

\newcommand{\N}{\mathbb{N}}

\renewcommand{\P}{\mathbb{P}}
\newcommand{\E}{\mathbb{E}}

\newcommand{\cW}{\mathcal{W}}

\newcommand{\cP}{\mathcal{P}}

\newcommand{\cA}{\mathcal{A}}

\newcommand{\eps}{\varepsilon}

\newcommand{\Z}{\mathbb{Z}}

\renewcommand{\emptyset}{\varnothing}
\renewcommand{\epsilon}{\varepsilon}
\renewcommand{\rho}{\varrho}
\renewcommand{\phi}{\varphi}

\newcommand{\Sh}{\mathbb{S}}

\renewcommand{\hat}{\widehat}

\renewcommand{\iint}{\int\hspace{-0.1in}\int}

\newcommand{\id}{\operatorname{id}}

\newcommand{\m}{m} 

\DeclareMathOperator{\diam}{diam}

\DeclareMathOperator{\supp}{supp}

\theoremstyle{plain}
\newtheorem{thm}{Theorem}[section]
\newtheorem{theorem}{Theorem}[section]

\newtheorem{lemma}[thm]{Lemma}
\newtheorem{prop}[thm]{Proposition}
\newtheorem{cor}[thm]{Corollary}

\theoremstyle{definition}

\newtheorem{definition}[thm]{Definition}

\newtheorem{remark}[thm]{Remark}

\subjclass[2010]{42A20 (Primary), 42A38, 37C45, 28A80, 60K05 (Secondary)}

\keywords{Fourier analysis, self-affine sets, Trigonometric series, Fourier series, random walk on groups, renewal theory, stationary measure}

\thanks{TS was partially supported by the Marie Sk{\l}odowska-Curie Individual Fellowship grant $\sharp$655310 and a start-up fund from the School of Mathematics, University of Manchester, UK}

\title{Fourier transform of self-affine measures}

\author{Jialun Li}
\address{Institute of Mathematics, University of Z\"urich,  Z\"urich, Switzerland}
\email{jialun.li@math.uzh.ch}

\author{Tuomas Sahlsten}
\address{School of Mathematics, Alan Turing Building, University of Manchester, Oxford Road, Manchester, UK}
\email{tuomas.sahlsten@manchester.ac.uk}

\begin{document}

\maketitle

\begin{abstract}
Suppose $F$ is a self-affine set on $\mathbb{R}^d$, $d\geq 2$, which is not a singleton, associated to affine contractions $f_j = A_j + b_j$, $A_j \in \mathrm{GL}(d,\mathbb{R})$, $b_j \in \mathbb{R}^d$, $j \in \mathcal{A}$, for some finite $\mathcal{A}$. We prove that if the group $\Gamma$ generated by the matrices $A_j$, $j \in \mathcal{A}$, forms a proximal and totally irreducible subgroup of $\mathrm{GL}(d,\mathbb{R})$, then any self-affine measure $\mu = \sum p_j f_j \mu$, $\sum p_j = 1$, $0 < p_j < 1$, $j \in \mathcal{A}$, on $F$ is a Rajchman measure: the Fourier transform  $\widehat{\mu}(\xi) \to 0$ as $|\xi| \to \infty$. As an application this shows that self-affine sets with proximal and totally irreducible linear parts are sets of rectangular multiplicity for multiple trigonometric series. Moreover, if the Zariski closure of $\Gamma$ is connected real split Lie group in the Zariski topology, then $\widehat{\mu}(\xi)$ has a power decay at infinity. Hence $\mu$ is $L^p$ improving for all $1 < p < \infty$ and $F$ has positive Fourier dimension. In dimension $d = 2,3$ the irreducibility of $\Gamma$ and non-compactness of the image of $\Gamma$ in $\mathrm{PGL}(d,\mathbb{R})$ is enough for power decay of $\widehat{\mu}$. The proof is based on quantitative renewal theorems for random walks on the sphere $\mathbb{S}^{d-1}$.
\end{abstract}

\section{Introduction and the main results}

\subsection{Spectrum of self-affine measures} \label{sec:spectrum}

Let $f_j = A_j + b_j$, $j \in \cA$, be a finite collection of affine contractions of $\R^d$ associated to non-singular matrices $A_j \in \mathrm{GL}(d,\R)$ with $\|A_j\| < 1$ and translation vectors $b_j \in \R^d$. The \textit{self-affine set} $F$ associated to $\{f_j : j \in \cA\}$ is the unique non-empty compact set $F \subset \R^d$ satisfying the invariance
$$F = \bigcup_{j \in \cA} f_j (F).$$
Moreover, a natural class of measures $\mu$ associated to $\{f_j : j \in \cA\}$ are the \textit{self-affine measures}, that is, those probability measures $\mu$ on $F$ satisfying $\mu = \sum_{j \in \cA} p_j f_j \mu$ for some weights $0 < p_j < 1$, $j \in \cA$, with $\sum_{j \in \cA} p_j = 1$, which appear in the dimension theory of self-affine sets. The geometry of self-affine sets and measures has been extensively studied since their introduction and popularisation after the work of Falconer \cite{Fa1, Fa2, Fa3}, see also the survey \cite{Fa4}. Recently, a useful connection to the dynamics of the stationary measure (Furstenberg measure) on the projective space has been developed in the study of self-affine sets. This was first introduced and popularised by the work of Falconer and Kempton \cite{FK} originally appeared in 2015 (related ergodic theoretic ideas were also simultaneously developed by B\'ar\'any \cite{barany}) and then it has been crucial in the analysis of the behaviour of self-affine sets and measures. See for example the recent works of B\'ar\'any, Hochman, Rapaport \cite{BHR}, B\'ar\'any-K\"aenm\"aki \cite{baranykaenmaki} and Feng \cite{feng} just to name a few. This paper follows this line of research but develops the connection to random walks on matrix groups further. In particular we will apply the recent advancements in the theory of random walks on reductive groups (see for example the book \cite{benoistquint} by Benoist and Quint) to study the \textit{spectral theory} of the self-affine sets and measures.  

Let $\mu$ be a probability measure on $\T^d=\R^d/\Z^d$. Consider the \textit{$L^2$ spectrum} $\sigma(\mu,L^2)$ of the multiplier $f \mapsto f \ast \mu$, in $L^2(\T^d)$, that is,
$$\sigma(\mu,L^2) = \overline{\{\widehat{\mu}(\m) : \m \in \Z^d\}} \subset \C,$$
where $\widehat{\mu}$ is the \textit{Fourier coefficient} of $\mu$, defined by
$$\widehat{\mu}(\m) := \int_{\T^d} e^{-2\pi i \m \cdot x} \, d\mu(x), \quad \m \in \Z^d.$$
The $L^2$ spectrum for singular measures $\mu$ in general has been a widely studied notion, in particular, depending on the behaviour of $\widehat{\mu}(\m)$ at infinity, it has various applications throughout analysis and geometry. The original motivation comes from Riemann's uniqueness problem \cite{Riemann, Cantor} of trigonometric series (see Section \ref{sec:uniqueness} below) where the asymptotic behaviour of $\widehat{\mu}$ is linked to the multiplicity and uniqueness of the support of the measure $\mu$. Moreover, the decay of $\widehat{\mu}$ at infinity can be linked to the prevalence of normal numbers or vectors in the support $\supp \mu$ of $\mu$, see the work of Davenport-Erd\"os-LeVeque \cite{DEL}, to the existence of arithmetic patterns in $\supp \mu$ \cite{LP09} and absolute continuity of fractal measures \cite{shmerkin}. Finally, in the harmonic analysis of singular measures $\mu$, H\"ormander \cite{H} proved that if $\widehat{\mu}(m) \to 0$ as $|m| \to \infty$, then $\sigma(\mu,L^2)$ also agrees with the $L^1$ spectrum of the multiplier $f \mapsto f \ast \mu$. This motivated (see the works of Sarnak \cite{Sarnak} and Sidorov-Solomyak \cite{SidorovSolomyak}) to study the $L^2$ spectrum of various specific singular measures $\mu$ such as the Cantor-Lebesgue measure and Bernoulli convolutions. More recently the decay results have been useful in establishing results on quantum resonances in quantum chaos, see for example Bourgain and Dyatlov \cite{BourgainDyatlov}.

For a probability measure $\mu$ on $\R^d$, we define its \textit{Fourier transform} by
\[\hat\mu(\xi)=\int e^{-2\pi i\xi\cdot x}\, d\mu(x)\quad \xi\in\R^d. \]
The heuristic idea behind the decay of Fourier transform at infinity is commonly explained by some ``chaotic properties'' within the singular measure $\mu$, such as when $\mu$ is given by a random measure associated to some random process such as Brownian motion or other random construction \cite{Kahane1,Kahane2,Kahane3,ShmerkinSuomala,FOS,FS} or by an equilibrium state to a sufficiently non-linear dynamical system, see the various recent works on these  \cite{Kaufman1,Kaufman2,JordanSahlsten,SahlstenStevens,BourgainDyatlov,Li1,Li2,LNP}. A classical result that really highlights this phenomenon is the theorem of Salem-Zygmund \cite{SZ} and Piatetski-Shapiro \cite{PS}, which says that the Cantor-Lebesgue measure $\mu_\lambda$ on the standard middle $\lambda$-Cantor set, $0 < \lambda < 1/2$, satisfies $\widehat{\mu}_\lambda(\xi) \to 0$ if and only if $\lambda^{-1}$ is not a Pisot number, that is, a real number whose powers approximate integers exponentially fast. A similar result also holds for Bernoulli convolutions by the work of Salem \cite{Salem} and Erd\"os \cite{Erdos}. Hence some sort of ``non-concentration'' to arithmetic progressions (lattices) should characterise Fourier decay at infinity. In a recent work \cite{LS} we developed this connection further and in the setting of general self-similar measures, we proved the Fourier decay of the self-similar measures as long as the random walk defined by the contractions does not concentrate on a lattice/arithmetic progressions. 

In the self-affine world we see that the correct analogue for the ``chaos'' assumption to gain Fourier decay of $\widehat{\mu}$ requires some form of irreducibility of the subgroup generated by the random matrix products from $A_j$, $j \in \cA$, in $\mathrm{GL}(d,\R)$. More formally this can be achieved if the subgroup generated by $A_j$, $j \in \cA$,
$$\Gamma = \langle A_j: j \in \cA \rangle < \mathrm{GL}(d,\R)$$
forms an irreducible and proximal group (see Section \ref{sec:rensphere} for precise definitions). In this case we can prove the following result on the spectrum of self-affine measures:

\begin{theorem}\label{thm:main}
Suppose $F$ is a self-affine set on $\R^d$, $d\geq 2$, which is not a singleton, associated to affine contractions $f_j = A_j + b_j$, $j \in \cA$. If $\Gamma = \langle A_j: j \in \cA \rangle$ forms a proximal and totally irreducible subgroup of $\mathrm{GL}(d,\R)$, then 
$$\widehat{\mu}(\xi) \to 0, \quad |\xi| \to \infty.$$
\end{theorem}

In the dimension theory of self-affine measures, especially in the recent work of B\'ar\'any-Hochman-Rapaport \cite{BHR} the same assumption on irreducibility of $\Gamma$ is required to prove the Hausdorff dimension of the self-affine measure $\mu$ agrees with its Lyapunov dimension.  In our setting, after an application of Cauchy-Schwarz inequality, we observe that the Fourier transform $\widehat{\mu}(\xi)$ of a self-affine measure $\mu$ reduces to a probabilistic expression depending on $\xi$, which appears commonly in renewal theory of random walks on the $d-1$ sphere $\mathbb{S}^{d-1}$. In these cases the irreducibility of $\Gamma$ is known to be crucial to establish a renewal theorem that proves Theorem \ref{thm:main}.

\subsection{Power decay of $\widehat{\mu}$ and the Zariski closure of $\Gamma$}\label{sec:power}

Theorem \ref{thm:main} does not say anything about the decay rate of $\widehat{\mu}$ at infinity. Having a quantitative rate of Fourier decay at infinity for $\mu$ can be important property in various applications. A classical application comes in harmonic analysis. Stein asked (see for example \cite[pp. 122-123]{Stein})  to characterise measures $\mu$ which are \textit{$L^p$ improving}, that is, those $\mu$ for which for some $1 < p < \infty$ there exists $r > p$ such that
$$f \ast \mu \in L^r(\R^d) \quad \text{for all } f \in L^p(\R^d).$$ 
In other words, the multiplier $f \mapsto f \ast \mu$ is a bounded operator from $L^p(\R^d) \to L^r(\R^d)$. One can use a complex interpolation argument to show that if a measure $\mu$ is $L^p$ improving for some $p$, then it is $L^p$ improving for all $1 < p < \infty$, see \cite[pp. 122-123]{Stein}. If a measure $\mu$ on $\R^d$ has a \textit{power Fourier decay} at infinity, that is, for some $\alpha > 0$ we have
$$|\widehat{\mu}(\xi)| = O(|\xi|^{-\alpha}), \quad |\xi| \to \infty,$$
then $\mu$ is $L^p$ improving with $r$ defined by $1/r = 1/p - \alpha$ and $p \leq 2 \leq r$, see Zygmund \cite[26, Vol. II, p. 1271]{Zygmund} for a proof.
Moreover, if one can prove the Fourier transform of $\mu$ has power decay at infinity, then the support of $\mu$ has positive \textit{Fourier dimension}, see the book \cite{Mattila} by Mattila for a history and survey of this notion in connection to various problems in geometric measure theory, fractal geometry and harmonic analysis. Positivity of the Fourier dimension of the support of $\mu$ implies  $\mu$ almost every vector in the support of $\mu$ is normal in every base, see for example \cite{DEL}. In the self-affine case we are considering establishing power Fourier decay of $\mu$ could be helpful in the study of absolute continuity for classes of self-affine measures, see the discussion in Section \ref{sec:abscont} below. 

In the proof of Theorem \ref{thm:main} we observe that the key point, where we obtain a slower rate than polynomial in the Fourier decay, comes from the rate of convergence for the renewal theorem for random walks on $\mathbb{S}^{d-1}$. In particular these rates can be improved when assuming $\R$-split for the Zariski closure of $\Gamma$, see Definition \ref{def:Rsplit} for a formal definition. Using the exponential speed in the renewal theorem established in \cite{Li2}, which is based on the discretized sum-product estimates invented by Bourgain and developed in \cite{hesaxce} and \cite{Li3}, we give a power decay.

\begin{theorem}\label{thm:mainquant}
Suppose $F$ is a self-affine set, which is not a singleton, associated to affine contractions $f_j = A_j + b_j$, $j \in \cA$. If the Zariski closure of $\Gamma = \langle A_j: j \in \cA \rangle$ is a connected $\R$-splitting reductive group acting irreducibly on $\R^d$, we have polynomial decay of Fourier transform: there exists $\alpha > 0$ such that $$|\widehat{\mu}(\xi)| = O(|\xi|^{-\alpha}), \quad |\xi| \to \infty.$$
In particular, $\mu$ is $L^p$ improving for all $1 < p < \infty$ and $F$ has positive Fourier dimension. 
\end{theorem}

\begin{remark} Aside from assuming $F$ is not a singleton, no other separation conditions are assumed in Theorems \ref{thm:main} and \ref{thm:mainquant} like the strong separation condition. However, if in Theorem \ref{thm:mainquant} the self-affine set $F$ satisfies the strong separation condition, then $\alpha$ can be made independent of the translations $b_j \in \R^d$, $j \in \cA$. This is similar to the idea behind the so called \textit{affinity dimension}, that is, the linear parts $A_j$ decide the dimension property of the self-affine sets or measures \cite{BHR}, see Remarks \ref{rem:s2} and \ref{rem:s2pf} below.
\end{remark}

Let us now discuss the validity of the assumptions of Theorem \ref{thm:mainquant} below:

\begin{remark} \begin{itemize}
	\item[(1)] Connectedness in Zariski topology and irreducibility imply strong irreducibility. Splitness and irreducibility imply proximality, which explains that the condition of Theorem \ref{thm:mainquant} is stronger than Theorem \ref{thm:main}.
	
	\item[(2)] The assumption on $\R$-splitting and connectedness of the Zariski closure of $\Gamma$ is always satisfied if the group $\Gamma$ is Zariski dense in $\mathrm{GL}(d,\R)$, which is a connected $\R$-splitting reductive group.
		
	\item[(3)] When $d=2$, the $\R$-splitting of the Zariski closure is actually equivalent to the condition of linear part in B\'ar\'any-Hochman-Rapaport \cite{BHR}, because a subgroup of $\mathrm{GL}(2,\R)$ whose image in $\mathrm{PGL}(2,\R)$ is non-compact and totally irreducible is always Zariski dense in $\mathrm{PGL}(2,\R)$. Due to the hypothesis that the linear part is contracting $\|A_j\|<1$, its Zariski closure is the whole $\mathrm{GL}(2,\R)$, which satisfies our assumption by the above remark.
	
		\item[(4)]Due to the structure of algebraic subgroups of $\mathrm{GL}(d,\R)$, for $d=3$ a totally irreducible subgroup whose image in $\mathrm{PGL}(3,\R)$ is not compact still always has a $\R$-splitting Zariski closure. Because the semisimple part of the connected component of the Zariski closure is conjugated to $\mathrm{SL}(3,\R)$ or $\mathrm{SO}(1,2)$. In the first case, due to $\|A_j\|<1$ the Zariski closure is the whole group $\mathrm{GL}(3,\R)$. In the second case, the Zariski closure is conjugated to $\R^*\times \mathrm{SO}(1,2)$, where $\R^*\times \mathrm{SO}(1,2)$ means $\R^*\id_3\times \mathrm{SO}(1,2)$ and which is $\R$-splitting and algebraically connected.  Here $\R^* = \R \setminus \{0\}$. Hence for $d=2,3$, Theorem \ref{thm:mainquant} holds under the condition of total irreducibility and the image in $\mathrm{PGL}(d,\R)$ is non-compact.
	
		\item[(5)]Starting from $d=4$, there is the algebraic subgroup $\mathrm{SO}(1,3)$ of $\mathrm{GL}(4,\R)$ which is not $\R$-splitting. The Zariski closure can also be algebraically non-connected, for example $\R^*\times\mathrm{O}(1,3)\simeq\R^*\times\mathrm{SO}(1,3)\rtimes \Z/2\Z$ whose algebraically connected component containing identity is $\R^*\times\mathrm{SO}(1,3)$. 
		
				\item[(6)]The group $\mathrm{SL}(2,\C)$ can be seen as a subgroup of $\mathrm{GL}(4,\R)$ and the action of $\mathrm{SL}(2,\C)$ on $\R^4$ is not proximal. Hence for $d=4$, the proximal condition in Theorem \ref{thm:main} is used to exclude this case. In dimension $d=4$, the situation becomes much more complicate due to the appearance of different type of Lie groups.
		\end{itemize}
\end{remark}

We conjecture that the power decay of the Fourier transform of the self-affine measure is still true under the condition of Theorem \ref{thm:main} (i.e. without the $\R$-splitting of the Zariski closure) or even without the proximal condition. For this purpose, we would need to generalise the renewal theorem with exponential error term to more general situation, which is a prospect for a future work.

\subsection{Uniqueness of multiple trigonometric series} \label{sec:uniqueness}

Let us now discuss a direct application of Theorem \ref{thm:main} to the uniqueness of trigonometric series. In classical Fourier analysis of functions $f : \T^d \to \R$, the trigonometric series associated to coefficients $a_\m$, $\m \in \Z^d$, are of the form
\begin{equation}\label{eq:trigsumser}\sum_{\m \in \Z^d} a_\m e^{2\pi i \langle x , \m \rangle}\end{equation}
and one asks if, for example, $a_\m = \widehat{f}(\m)$, the Fourier coefficient of $f$, then do series \eqref{eq:trigsumser} converge to $f$ and in what sense (pointwise, $L^2$, and so on). Thus it is natural to ask about the \textit{uniqueness} of such trigonometric series, that is, what sets $F \subset \T^d = \R^d / \Z^d$ satisfy the following property: If $a_\m,b_\m \in \C$, $\m \in \Z^d$, are chosen such that
\begin{align}\label{eq:uniq}\sum_{\m \in \Z^d}a_\m e^{2\pi i \langle x , \m \rangle} = \sum_{\m \in \Z^d} b_\m e^{2\pi i \langle x , \m \rangle},\end{align}
for all $x \in \T^d \setminus F$, then $a_\m = b_\m$ for all $\m \in \Z^d$. Any set $F$ satisfying the Uniqueness Problem is called a set of \textit{uniqueness}. Otherwise $F$ is a set of \textit{multiplicity}. In dimension $d = 1$ this problem originated in the seminal works of Riemann, Cantor and Young \cite{Riemann,Cantor,Young}. In these works it is proved that any countable set in $\R$ is a set of uniqueness. After these works the problem has become a popular topic in Fourier analysis, see for example the survey of K\"orner \cite{korner} and the references therein.

In the higher dimensional setting there are various ways to generalise the uniqueness problem, and depending on the way we sum, one has very different outcomes, see \cite{AshWang,AshWang2} for discussion and references. If the summation in \eqref{eq:uniq} is \textit{rectangular}, that is, we sum over $\m \in R_n$ of boxes $R_n \to \Z^d$ as $n \to \infty$, i.e. $R_n = \prod_{j = 1}^d B_{\Z}(0,r_j^n)$ with $r_j^n \to \infty$ as $n \to \infty$ for all $j =1,\dots,d$, then Ash, Freiling and Rinne \cite{AFR} proved that $F = \emptyset$ is a set of (rectangular) uniqueness. Another way to generalise this is to consider \textit{spherical summation} in \eqref{eq:uniq}, that is, sum over $\m \in B_{\Z^d}(0,r_n)$ for some radii $r_n \to \infty$ as $n \to \infty$. In this case Bourgain \cite{Bourgain1996} established that $F = \emptyset$ is a set of (spherical) uniqueness and later Ash and Wang \cite{AshWang} generalised this to all finite sets $F$. See Ash's survey \cite{Ash} for more historical overview and references of the uniqueness problem in recent literature.

For the uncountable (fractal) case, in dimension $d = 1$, the works of Salem, Zygmund \cite{SZ} et al. have been attempting to give classifications of the sets of uniqueness and multiplicity. The known fractal examples of sets of multiplicity were the middle $\lambda$ Cantor sets $C_\lambda$ with $0 < \lambda < 1/2$ where the middle $1-2\lambda$ part of $[0,1]$ is removed and that $\lambda^{-1}$ is not a Pisot number. Moreover, recently by using methods from random walks on the additive group $\R$, we proved in \cite{LS} that every self-similar set on $\R$, which is not a singleton and the contractions $f_j(x) = r_j x + b_j$ defining $F$ satisfy that $\log r_j / \log r_\ell$ is irrational for some $j \neq \ell$, then $F$ is a set of multiplicity.  In all of these cases, proving the multiplicity of a compact set $F$ is closely related to the spectrum $\sigma(\mu,L^2)$ of measures $\mu$ supported on $F$. In particular, using Menshov's localisation argument, see for example the survey of K\"orner \cite{korner}, if a probability measure $\mu$ on $F$ satisfies $\widehat{\mu}(\m) \to 0$ as $|\m|\to \infty$, then for any sequence of rectangles $R_n \to \Z^d$, as $n\to \infty$, we have
\begin{align}\label{eq:fourconv}\lim_{n \to \infty}\sum_{\m \in R_n}\widehat{\mu}(\m) e^{2\pi i \langle x , \m \rangle} = 0\end{align}
for any $x \notin F$. Here we note that \cite[Theorem 3]{korner} is stated only for one dimension, but any rectangular Fourier series is a product of $d$ one dimensional Fourier series and restriction of $\mu$ on these spaces has also Fourier decay. Hence any $F \subset \T^d$ supporting a probability measure $\mu$ with $\widehat{\mu}(\m) \to \infty$, as $|\m| \to \infty$, must be a set of (rectangular) multiplicity. Thus as a Corollary of Theorem \ref{thm:main} we have

\begin{cor}\label{cor:multi}
Suppose $F$ is a self-affine set on $\T^d$, which is not a singleton, associated to affine contractions $f_j = A_j + b_j$, $j \in \cA$.  If $\Gamma = \langle A_j: j \in \cA \rangle$ forms a proximal and totally irreducible subgroup of $\mathrm{GL}(d,\R)$, then $F$ is a set of rectangular multiplicity.
\end{cor}

Here, the torus $\T^d$ is identified with the cube $[0,1)^d$ in $\R^d$ and $F$ is contained in $[0,1)^d$.
Corollary \ref{cor:multi} leaves open the opposite case: if we assume that the group $\Gamma$ preserves some proper subspace of $\R^d$, that is, when $\Gamma = \langle A_j: j \in \cA \rangle$ lacks irreducibility, then does this imply actually that $F$ is a set of (rectangular) uniqueness? This would be the analogue of result of Salem-Zygmund \cite{SZ} and Piatetski-Shapiro \cite{PS} that the middle $\lambda$-Cantor set is a set of multiplicity if and only if $\lambda^{-1}$ is not a Pisot number. 
However, as far as we know this has not been developed further after this so we conjecture a self-affine set is a set of (rectangular) uniqueness if the assumptions of Theorem \ref{thm:main} fail. Another direction where to look at Corollary \ref{cor:multi} would be to consider the spherical summation in \eqref{eq:uniq} instead of the rectangular one. Here one faces some obstacles as it is not clear how to deduce a spherical analogue of \eqref{eq:fourconv} from the decay of $\widehat{\mu}$ at infinity. For example, in Weisz's survey \cite[Page 27]{W} on multiple trigonometric series we see that spherical summations do not have such nice approximation formulae in higher dimensions as one dimension used in Menshov's localisation argument for $\mu$.

\subsection{Absolute continuity of self-affine measures} \label{sec:abscont}

A motivation for establishing the power decay of Fourier transform of a fractal measure $\mu$ appears often in the study of absolute continuity of $\mu$ such as Bernoulli convolutions arising from overlapping self-similar iterated function systems, see for example the works of Shmerkin \cite{shmerkin}, Shmerkin-Solomyak \cite{ShmerkinSolomyak} and Saglietti, Shmerkin and Solomyak \cite{SSS}. See also the recent work on absolute continuity of self-similar measures in dimension at least $3$ by Lindenstrauss and Varju \cite{LindenstraussVarju}, which is more closely related to our setting.

In Shmerkin's original work \cite{shmerkin} on the absolute continuity of Bernoulli convolutions, the core idea is to separate the Bernoulli convolution into the convolution of two self-similar measures. The additive convolution structure of the self-similar measure is crucial in this method, which enables us to combine the full dimension and Fourier decay to obtain the absolute continuity (also known as \textit{Erd\"os-Kahane method}, see \cite{shmerkin} for a good survey of the topic). The method of Saglietti, Shmerkin and Solomyak in \cite{SSS} on the absolute continuity of non-homogeneous self similar measure is to try to transfer every thing to homogeneous self-similar case and it is done by using the commutativity of $\R$. For example, if we have two map with different contraction ratios $r_1 \neq r_2$, then the twice iteration has four maps and only three different contraction ratios $r_1^2,r_2^2,r_1r_2$, with two different maps with the same contraction ratio $r_1r_2$. In higher dimensional self-similar case, Lindenstrauss and Varju \cite{LindenstraussVarju} also used the same idea, that is, using the commutativity of $\R$, to extract a part of the IFS with the same contraction ratio.

However, in the self-affine case we are considering, it is the non-commutativity of the linear parts that gives the power Fourier decay of the self-affine measure in Theorem \ref{thm:mainquant} by the renewal theorem proved on $\Sh^{d-1}$. We are unable to extract a part of the IFS with the same contraction ratio, so the methods used by Saglietti, Shmerkin and Solomyak \cite{SSS} and Lindenstrauss and Varju \cite{LindenstraussVarju} cannot be adapted in our setting. Any kind of convolution structure of the self-affine measure is difficult to find in this case, which would probably require a new method involving different type of separation of the self-affine measure. We expect that there should be a result for self-affine measures analogous to the work of Lindenstrauss and Varju \cite{LindenstraussVarju}, say, a result with a parametrised family of self-affine measures or a random version saying the absolute continuity of $\mu$ holds almost surely, where Theorem \ref{thm:mainquant} would be applied to. 

\subsection*{Acknowledgements} We thank Jean-Fran\c{c}ois Quint for useful discussions related to the absolute continuity of the stationary measure. The second author also thanks Elon Lindenstrauss for useful discussion about the paper \cite{LindenstraussVarju} back in 2014 while visiting the Hebrew University of Jerusalem. Part of this work was conducted while the second author was visiting Institut de Mathématiques de Bordeaux in January 2019, and the authors would like to thank the hospitality of the institution.

\section{Preliminaries on self-affine geometry}

\subsection{Symbolic notations and products of matrices}

Fix an iterated function system $\{f_j : j \in \cA\}$ consisting of maps $f_j = A_j + b_j$ for $A_j \in \mathrm{GL}(d,\R)$ and $b_j \in \cA$ with $\|A_j\| < 1$ for all $j \in \cA$, where we write $\|\cdot\|$ as the operator norm of matrices. Let $F$ be the compact non-empty self-affine fractal associated to $\{f_j : j \in \cA\}$, that is,
$$F = \bigcup_{j \in \cA} f_j (F).$$
We say that \textit{$\mu$ is a self-affine measure on $F$} if there exist weights $0 < p_j < 1$, $j \in \cA$, such that $\sum_{j \in \cA} p_j = 1$, and $\mu$ satisfies the relation
$$\mu = \sum_{j \in \cA} p_j f_j \mu.$$

\begin{definition}[Word spaces $\cA^n$, $\cA^*$]
Write $\cA^*$ the set of all finite words and $\cA^n$ the set of all length $n$ words. Let $w \in \cA^n$. Define the composition
$$f_w := f_{w_1} \circ \dots \circ f_{w_n} = A_w + b_w,$$
where
$$A_w := A_{w_1}\dots A_{w_n}$$
and $b_w$ is the corresponding translation component. 

\subsection{Regularity of self-affine measures} Self-affine measures enjoy the following weak form of Ahlfors-David regularity, which is a folklore result that we prove here. Currently we found a reference by Feng and Lau \cite{FengLau}, where this was proven for self-similar measures on $\R^d$, and we follow the same idea. 	We only need the upper bound in our proof and the upper bound can also be found in Aoun-Guivarc'h \cite{aounguivarch}, where they prove regularity of the stationary measures which contains our self-affine measures. 

The following notation is only used in this subsection. We write
$$F_w := f_w (F), \quad w \in \cA^*.$$
Note that the diameter $\diam(F_w) \leq D\|A_w\|$, where $D$ is the diameter of $F$.
\end{definition}

Denote
$$r_w := \|A_w\|,$$
which will shrink exponentially as $|w| \to \infty$ due to $\|A_j\| < 1$ for all $j \in \cA$.

\begin{definition}[Words $\cA_r$ with prescribed matrix norm] For $r > 0$ write the collection of words $w$ corresponding to roughly norm $r$ matrices $A_w$ as:
$$\cA_r := \{w \in \cA^* : r_w < r \leq r_{\tilde w}\},$$
where $\tilde w$ is the length $|w| - 1$ word obtained from $w$ by removing the last letter of $w$. 
\end{definition}

Note that if we fix small enough $r > 0$, then due to the definition of self-affine measure, we have the following invariance:
\begin{align}\label{eq:aridentity}\mu = \sum_{w \in \cA_r} p_w f_w \mu.\end{align}

\begin{lemma}\label{lma:frostman}
Assume the self-affine set $F$ is not a singleton and $\mu$ is a self-affine measure on $F$. Then there exist constants $C_1,C_2 > 0$, $r_0 > 0$, and exponents $0 < s_2 < s_1$ such that
$$C_1 r^{s_1} \leq \mu(B(x,r)) \leq C_2 r^{s_2}, \quad x \in F, \quad 0 < r \leq r_0.$$
\end{lemma}

\begin{proof}
Since $F$ is not a singleton, we can find two words $w^1,w^2 \in \cA_\eta$ for some $0 < \eta \leq 1$ such that $F_{w^1} \cap F_{w^2} = \emptyset$. Since $F_{w^1}$ and $F_{w^1}$ are compact, we can find $0 < r_0 < r_{\min} := \min\{r_j : j \in \cA\}$ such that for any $x \in \R^d$ the ball $B(x,r_0)$ intersects at most one of the two sets. Write
$$\phi(r) = \sup_{x \in \R^d} \mu(B(x,r)), \quad 0 < r \leq r_0$$
and denote
$$c :=1/ \max\{\|A_w^{-1}\| : w \in \cA_\eta\}.$$
Thus for any $x \in \R^d$ and $0 < r < r_0$ we have either that $B(x,r) \cap F_{w^1} = \emptyset$ or $B(x,r) \cap F_{w^2} = \emptyset$. If $B(x,r) \cap F_{w^1} = \emptyset$, then by the identity \eqref{eq:aridentity} over the words in $\cA_r$ we have
\begin{align*}
\mu(B(x,r)) &= \sum_{w \in \cA_r, F_w \cap B(x,r) \neq \emptyset} p_w \mu(f_w^{-1}(B(x,r)) \\
& \leq \sum_{w \in \cA_r, w \neq w^1} p_w \mu(f_w^{-1} B(x,r)) \\
& \leq \sum_{w \in \cA_r, w \neq w^1} p_w \phi(r/c) \\
& = (1-p_{w^1}) \phi(r/c).
\end{align*}
If $B(x,r) \cap F_{w^2} = \emptyset$, a symmetric argument also shows 
$$\mu(B(x,r)) \leq (1-p_{w^2}) \phi(r/c).$$
Hence we have proved for $t = \max\{1-p_{w^1}, 1-p_{w^2}\}$ that $\mu(B(x,r)) \leq t \phi(r/c)$ so we have the doubling condition
$$\phi(r) \leq t \phi(r/c), \quad 0 < r \leq r_0.$$
This gives the upper bound we claimed using 
$$C_2 = \phi(r_0)/(tr_0^{s_2}) \quad \text{and} \quad s_2 = \frac{\log t}{\log c},$$ 
since for all $0 < r \leq r_0$ by choosing $n \in \N$ such that $c^n r_0 < r \leq c^{n-1} r_0$ we obtain:
$$\mu(B(x,r)) \leq \mu(B(x,c^{n-1} r_0)) \leq \phi(c^{n-1} r_0) \leq t^{n-1} \phi(r_0) = \frac{\phi(r_0)}{tr_0^{s_2}} (c^nr_0)^{s_2} \leq C_2 r^{s_2}.$$

For the lower bound, let $D = \diam(F)$. Then for every $x \in F$ and $0 < r \leq r_0$ there exists a word $w \in \cA_{r/D}$ such that $x \in F_w$. Thus diameter $\diam(F_w) < r$ and $F_w \subset B(x,r)$. This gives us
$$\mu(B(x,r)) \geq \mu(F_w) \geq p_w = (r_w)^{\frac{\log p_w}{\log r_w}} \geq \Big(\frac{r_{\min} r}{2D}\Big)^{\frac{\log p_w}{\log r_w}} \geq C_1 r^{s_1}$$
with $s_1 = \max\{\log p_j / \log r_j : j \in \cA\}$ and $C_1 = \min\{1,r_{\min} / (2D)\}^{s_1}$.

\end{proof}
\begin{remark}\label{rem:s2}
	When the self-affine set $F$ satisfies the \textit{strong separation condition}, that is $F_j \cap F_\ell = \emptyset$ for all $j \neq \ell$, then we can take $\eta=1$ in the above proof and $w_1,w_2$ be two different single words. Then the constants $t,c$ are independent of the translation part $b_j$, hence the exponent $s_2$ is also independent of the translation part. 
\end{remark}

\section{Quantitative renewal theorem for random walks on the sphere $\mathbb{S}^{d-1}$}\label{sec:rensphere}
\def\lf{{\mathrm{Lip}}}
\def\ck{{\mathrm{Lip}}}
\def\Car{\mathrm{E}_\mathrm{C}\,}

The main method in Theorem \ref{thm:main} and Theorem \ref{thm:mainquant} is to reduce the analysis to a renewal theorem (and their quantitative rates) for random walks on the sphere $X:= \Sh^{d-1}$. There are pioneer work of Kesten \cite{kesten} for general Markov process, and recent works of Guivarc'h-Le Page \cite{GLP} and Boyer \cite{Boyer} for the same situation as ours.

Here we need a quantitative version. A similar situation is the renewal theorem for random walks on the projective space $\P(\R^d)$, which was done by Li \cite[Proposition 4.17]{Li1}. Here we will give the analogous renewal theorem in $X$ and later in the paper describe the modifications we need to introduce to the proof of \cite[Proposition 4.17]{Li1} to get the renewal theorem on $X$. 

\begin{definition}[Irreducible and proximal subgroups]Let $\Gamma$ be a subgroup of $\mathrm{GL}(d,\R)$.
\begin{itemize}
\item[(1)] We call $\Gamma$ \textit{strongly irreducible} (totally irreducible) if the group $\Gamma$ does not fix a non-empty union of a finite number of proper subspaces of $\R^d$. 
\item[(2)] The subgroup $\Gamma$ is \textit{proximal}, if there exists an element $g$ in $\Gamma$ such that $g$ has a unique eigenvalue of greatest absolute value among all the eigenvalues of $g$, and this eigenvalue is simple. 
\end{itemize}
These two conditions are important in the theory of products of random matrices. If we want to know whether a subgroup satisfies these condition, then it is sufficient to check the Zariski closure (see \cite{margulis}), which is much simpler.
\end{definition}

We fix an euclidean norm on $\R^d$. For an element $g$ in $\mathrm{GL}(d,\R)$, define the norm cocycle by 
$$\sigma(g,z):=\log\frac{|gv|}{|v|}$$ 
for $z$ in $X$ and $v\in \R^+ z$, where $\R^+ z = \{tz : t > 0\}$. We will write $gz$ for a point in $X$ given by $gv/|gv|$ for $v\in\R^+z$, which gives an action of $G$ on $X$. This notation should not be confused with $gv$ which is a vector in $\R^d$.

Let $\lambda$ be a Borel probability measure on $\mathrm{GL}(d,\R)$ with compact support such that $\Gamma_\lambda$, the group generated by the support of $\lambda$, acts proximally and strongly irreducibly on $\R^d$. From now on, we will always keep these assumptions on $\lambda$.
A Borel probability measure $\nu$ on $X$ is called $\lambda$-stationary if
\[\nu=\lambda*\nu:=\int g_*\nu \,d \lambda(g). \]
Let $\nu$ be a $\lambda$-stationary measure on $X$. Equip $X$ with the induced distance from the euclidean distance on $\R^d$. One important regularity of the stationary measure is the following
\begin{lemma}[Guivarc'h regularity]\label{lma:guivarch}
	Suppose $\nu$ is a $\lambda$-stationary measure. Then there exist $C,\alpha>0$ such that for every hyperplane $Y$ in $X$ and $r>0$,
	\[\nu(x \in X : d(x,Y)\leq r)\leq Cr^\alpha. \]
\end{lemma}
\begin{proof}
	Let $\pi$ be the projection from the sphere $X$ to the projective space $\P(\R^d)$. Then the pushforward measure $\pi_*(\nu)$ is the Furstenberg measure on $\P(\R^d)$. The regularity of $\nu$ comes from the regularity of the Furstenberg measure (See \cite{guivarch} or Theorem 14.1 in \cite{benoistquint}).  
\end{proof}
Let $\sigma_{\lambda}$ be the first Lyapunov constant of $\lambda$, which is given by 
\[\sigma_{\lambda}=\int_G\int_X \sigma(g,x)\,d\lambda(g)\,d\nu(x)=\lim_{n\rightarrow+\infty}\frac{1}{n}\log\|g_1\cdots g_n\| \]
almost surely, where $g_i$ are i.i.d. random variables with the same distribution $\lambda$, see \cite[Theorem 4.28]{benoistquint} or \cite{Fur}. 
\subsection{Stopping time for the random walk} \label{sec:stoppingtime}
Set $V=\R^d$, $d\geq 2$. Recall that $\lambda$ is a Borel probability measure on $\mathrm{GL}(V)$ with compact support such that $\Gamma_\lambda$ acts proximally and strongly irreducibly on $V$. Suppose that every element $g$ in $\supp\lambda$ satisfies $\|g\|<1$.
Let $X_1,X_2,\dots$ be the random variables taking values in $\mathrm{GL}(V)$ with the same distribution $\lambda$.

Define the matrix product
$$S_n = X_n\dots X_2X_1.$$
Because the operator norm of $X_j$ is less than $1$ almost surely, the norm of $S_n$ decrease with respect to $n$.
For $t > 0$ and $x$ in $ X$ define a stopping time $n_t :  X \to \N$ by 
$$n_t(x) = \inf\{n \in \N : -\sigma(S_n,x)>t\}.$$
The advantage of $X_n\cdots X_1x$ with respect to $X_1\cdots X_nx$ is that the limit distribution is simpler to understand. The action is just multiplying a matrix in the left.

We define a renewal operator for stopping time for $x\in X$ and $t>0$
\[ \E_t f(x)=\E f(S_{n_t(x)}x,\sigma(S_{n_t(x)},x)+t)=\sum_{n\geq 0}\int_{\sigma(g,x)\geq - t> \sigma(hg,x) } f(hgx,\sigma(hg,x)+t)\,d\lambda(h) \,d\lambda^{*n}(g), \]
where $\lambda^{\ast n}$ is the $n$-fold self-convolution of $\lambda$, defined by $\lambda^{\ast n} = \lambda \ast \lambda^{\ast(n-1)}$ for $n \geq 1$ and $\lambda^{\ast 0}$ is the Dirac mass on the identity matrix.

We need to study the random walk and $\lambda$-stationary measures on the sphere $X$ instead of projective spaces $\P(V)$. Recall that a convex cone in $V$ is called proper if it does not contain a line. Guivarc'h and Le Page \cite[Proposition 2.14]{GLP} proved that if $\Gamma_\lambda$ preserves a proper convex cone in $V$ then there exists two $\lambda$-stationary $\lambda$-ergodic measures $\nu_1,\nu_2$ on the sphere $X$. There will be two continuous positive function $p_1$ and $p_2$ (For the characterization of $p_1$ and $p_2$, see \cite{Boyer} Lemma 2.13) on $X$ such that $p_1+p_2=1$, $p_i|{\supp \nu_j}=\delta_{i,j}$, where $\delta_{i,j}$ is the Kronecker symbol, and for $j=1,2$, $x\in X$
\[ p_j(x)=\int p_j(gx)\,d\lambda(g). \]
On the contrary, if $\Gamma_\lambda$ does not preserve any proper convex cone in $V$, then the $\lambda$ stationary measure on $ X$ is unique.

Let us now write formally the renewal theorem in our situation. For this purpose define the following measures $\nu_x$:

\begin{definition}\label{defi:nux}
	For $x\in X$, we define
	\[\nu_x :=p_1(x)\nu_1+p_2(x)\nu_2 \text{ in the first case,  otherwise } \nu_x=\nu.\]
\end{definition}
These measures $\nu_x$ are the limit distributions for the random walk on $X$ starting from $x$, following the law of $\lambda$. When we are given a Borel probability measure $\tau$ with compact support $\supp \tau$, the \textit{size of the support of $\tau$} is defined by the maximal norm:
	\[|\supp\tau|:=\sup\{\|g\| : g\in\supp\tau\}. \]

In the renewal theorem, we need to assume some regularity from the test functions, so for this purpose let us define the \textit{Lipschitz norm} of $f : X \times \R \to \C$ by  
	\[\|f\|_{\mathrm{Lip}}=\|f\|_\infty+\sup_{(x,v)\neq(x',v')}\frac{|f(x,v)-f(x',v')|}{d(x,x')+|v-v'|}.\]
	Using this notation, we have the following:

\begin{prop}[Renewal theorem irreducibility and proximality]\label{prop:stop} 
	Let $\lambda$ be a Borel probability measure on $\mathrm{GL}(V)$ with compact support, such that the group $\Gamma_\lambda$ acts proximally and strongly irreducibly on $V$. Suppose that every element $g$ in $\supp\lambda$ satisfies $\|g\|<1$.
	
	Let $ f$ be a continuous function on $X\times\R$ with $\| f\|_\lf$ finite. Then for $t>0$ and $x\in X$ we have
	\begin{align*}
	\E_t f(x) = \,\, & \frac{1}{|\sigma_{\lambda}|}\int_{X}\int_G\int^{-\sigma(h,y)}_{0} f(hy,\sigma(h,y)+u)\,\,d u\,\,d\lambda(h)\,\,d\nu_x(y)+o_t\| f\|_\lf,
	\end{align*}
	where $o_t$ tends to zero as $t$ going to $\infty$ and $o_t$ does not depend on $f$ and $x$.
\end{prop}

The proof of Proposition \ref{prop:stop} is postponed to Section \ref{sec:renewal} later.

\begin{remark}
	\begin{itemize}
	\item[(1)] The key point in renewal theorem is that the limit distribution of the jump $t+\log |S_{n_t}(x)|$ is absolute continuous. Our assumption of proximality and irreducibility are used to obtain non-arithmeticity, which is the crucial condition to obtain continuous limit distribution.
	
	\item[(2)] The above renewal theorem for stopping time without error term has already been proved in \cite[Theorem 4.8]{GLP}, by using the method in \cite{kesten}.
	
	\item[(3)] In our application, a quantitative version is needed. Proposition \ref{prop:stop} is only proved for $\mathrm{SL}_2(\R)$ in \cite{Li1} by using the method of transfer operator developed by Guivarc'h. The same strategy works under strongly irreducible and proximal condition. This is the most technical part, please see Section \ref{sec:renewal}.
	\end{itemize}
\end{remark}

If we know that the Zariski closure of the group $\Gamma$ is a connected $\R$-split reductive group, then we can use the spectral gap established in \cite{Li2} to obtain an exponential error term in the above Proposition \ref{prop:stop}. Recall the definition of real splitting reductive groups (\cite{borel}).

\begin{definition}[$\R$-split reductive groups]\label{def:Rsplit}
Let $G$ be an algebraic group and $\mathfrak{g}$ be its Lie algebra. We call $G$ real reductive if the group $G$ is defined over $\R$ and if the unipotent radical, the maximal connected nilpotent normal subgroup of $G$, is trivial. 

We call $G$ a $\R$-split reductive group if $G$ is real reductive and if there is a Cartan subgroup of $G$ splitting over $\R$, which means that $G$ contains a subgroup $G_0$ isomorphic to $(\R^*)^n$, where $\R^* = \R \setminus \{0\}$, such that the adjoint representations $\{\mathrm{Ad}(g) : g \in G_0 \}$ are simultaneously diagonalisable, and $G_0$ has finite index in the normalizer.
\end{definition}

Assuming in addition to the assumptions of Proposition \ref{prop:stop} the $\R$-split and connectedness for the Zariski closure for $\Gamma$, we can state the following quantitative renewal theorem using exponential error terms. Recall that the Zariski closure of a strongly irreducible group $\Gamma$ is always reductive, which explains our condition in the following proposition. A Borel probability measure $\lambda$ on an algebraic group $G$ is called Zariski dense if the group $\Gamma_\lambda$ is Zariski dense in $G$.

\begin{prop}[Renewal theorem for $\R$-splitting Lie groups on the sphere]\label{prop:stopexp}
	Let $G$ a connected reductive group defined and split over $\R$, which acts irreducibly on $V$.
	Let $\lambda$ be a Zariski dense Borel probability measure on $G$ with a compact support. Suppose that every element $g$ in $\supp\lambda$ satisfies $\|g\|<1$. 
	
	There exists $\eps_1>0$ such that the following holds. Let $f$ be a smooth function on $ X\times \R$. Then for $t>0$ and $x\in X$, we have
	\begin{align*}
	\E_t f(x) = \,\, &\frac{1}{|\sigma_{\lambda}|}\int_{X}\int_G\int^{-\sigma(h,y)}_{0} f(hy,\sigma(h,y)+u)\,d u\,d\lambda(h)\,d\nu_x(y)\\
	&+e^{-\eps_1 t}O(e^{4\eps_1|\supp f|}(\|f\|_{\mathrm{Lip}}+\|\partial_{uu}f\|_{\mathrm{Lip}})).
	\end{align*}
\end{prop}

The proof of Proposition \ref{prop:stopexp} is also done in the later Section \ref{sec:renewal}. We will now show how to apply Proposition \ref{prop:stop} and Proposition \ref{prop:stopexp} to prove Theorem \ref{thm:main} and Theorem \ref{thm:mainquant}. We will postpone the proofs of renewal theorems to the later sections.

\section{Proof of the main results}

\subsection{Strategy of the proof} Let us now prove the main results, namely, Theorem \ref{thm:main} and Theorem \ref{thm:mainquant}.
Define the probability measure $\lambda_1$ on $\mathrm{GL}(d,\R)$ by
$$\lambda_1 = \sum_{j \in \cA} p_j \delta_{ A_j}.$$
Let $A_{w_1}, \dots, A_{w_n}$ be an i.i.d. sequence of matrices distributed according to $\lambda_1$ from $\mathrm{GL}(d,\R)$ and for $w \in \cA^n$ define the product
$$A_w := A_{w_1} \dots A_{w_n}.$$
Because the operator norm of $A_j$ is less than $1$, the norm of $A_w$ will decrease with respect to $n$. Let $\tilde{\cA}$ be the set of infinite words $w_1w_2\cdots $, equipped with the measure $\lambda_1^{\otimes\N}$.
Using the stopping time notation, recall Section \ref{sec:stoppingtime}. For $t > 0$ and $z \in X=\mathbb{S}^{d-1}$ we define the function $n_t : X\times \tilde{\cA} \to \N$ by 
$$n_t^w(z) := \inf\{n \in \N : -\sigma((A_{w_1}\cdots A_{w_n})^\top, z)>t\},$$
where $w$ is an element in $\tilde{\cA}$.
Then $n_t(z)$ is a stopping time on trajectory space $\tilde{\cA}$. Write
$$\cW_t(z) = \{w_1\cdots w_{n^w_t(z)}: w \in \tilde{\cA} \},$$
which is a finite subset of $\cA^*$,
and define $\P_t = \P_t^z$ on $\cA^*$ by setting $\P_t = \sum_{w \in \cW_t(z)} p_w\delta_{w}$, where $\delta_w$ is the Dirac measure on $w$. Then in particular the self-affinity of $\mu = \sum_j p_j f_j \mu$ implies the following:

\begin{lemma}\label{lma:splitstoptime} For any $z \in  X$ and $t > 0$ we have
$$\mu = \E_{\P_t^z} f_w \mu = \sum_{w \in \cW_t(z)} p_w f_w \mu.$$
\end{lemma}

For simplicity of the notation, we will sometimes abbreviate $\cW_t(z) $ to $\cW_t$ if there is no ambiguity. In fact, once the direction of $\xi$ is fixed, the stopping time is fixed.

\subsection{Reduction to matrix products}

Fix $\xi \in \R^d$, $z = \xi / |\xi| \in  X$ and $t > 0$ with the associated stopping time $n_t(z)$. We first reduce the Fourier transform to an expression involving products 
$$A_w = A_{w_1} \dots A_{w_n}$$
of the matrices by using the stopping time $n_t(z)$. As an application of Jensen's inequality we obtain:

\begin{lemma}\label{lma:cauchy}
For all $\xi \in \R^d$ and $t > 0$ we have
$$|\widehat{\mu}(\xi)|^2 \leq \iint  \sum_{w \in \cW_t(z)} p_w  e^{-2\pi i A_w^\top \xi\cdot (x-y)} \, d\mu(x) \, d\mu(y), $$
with $z=\xi/|\xi|$.
\end{lemma}

\begin{proof} 
By Lemma \ref{lma:splitstoptime}, we obtain
$$\widehat{\mu}(\xi) = \sum_{w \in \cW_t} p_w \int e^{-2\pi i \xi \cdot f_w(x)} \, d\mu(x),$$
where the product weight
$$p_w := p_{w_1}\dots p_{w_n}.$$
Thus by Jensen's inequality or Cauchy-Schwarz's inequality, we have
$$|\widehat{\mu}(\xi)|^2 \leq \sum_{w \in \cW_t} p_w \Big|\int e^{-2\pi i \xi \cdot f_w(x)} \, d\mu(x)\Big|^2.$$
Opening up we see that
\begin{align*}
\sum_{w \in \cW_t} p_w \Big|\int e^{-2\pi i \xi\cdot f_w(x)} \, d\mu(x)\Big|^2 &= \iint  \sum_{w \in \cW_t} p_w  e^{-2\pi i \xi\cdot (f_w(x)-f_w(y))} \, d\mu(x) \, d\mu(y). 
\end{align*}
Here by definition we have that
$$f_w(x)-f_w(y) =  A_w (x-y),$$
so the proof is complete by $\xi\cdot A_w(x-y)=A_w^\top \xi\cdot (x-y)$.
\end{proof}

\subsection{Controlling nearby points}
Fix $s = s(\xi),t = t(\xi) > 0$ (which will be specified later) such that
$$|\xi|=se^t.$$ 
For $\delta=s^{-\eps}$ with $\eps=1/10$ denote the tube
$$A_\delta = \{(x,y) \in \R^d \times \R^d : |x-y| \leq \delta\}.$$
Here we use the upper Frostman property (Lemma \ref{lma:frostman}) to control this part.

\begin{lemma}\label{lma:expsum}
We have for some $\beta > 0$ and $C > 0$ that for all $\delta < r_0$ from Lemma \ref{lma:frostman} the following decay holds:
$$ \Big|\iint_{A_\delta}  \sum_{w \in \cW_t(z)} p_w  e^{-2\pi i \xi\cdot A_w (x-y)} \, d\mu(x) \, d\mu(y)\Big| \leq C\delta^\beta.$$
\end{lemma}

\begin{proof}
Since $\sum_{w \in \cW_t} p_w = 1$, we have that
$$\Big|\iint_{A_\delta}  \sum_{w \in \cW_t} p_w  e^{-2\pi i \xi\cdot A_w (x-y)} \, d\mu(x) \, d\mu(y)\Big| \leq (\mu \times \mu) (A_\delta).$$
Thus by Fubini's theorem and Lemma \ref{lma:frostman} we have the following upper bound (for $C = C_2$ and $\beta = s_2$) due to $\delta < r_0$ in Lemma \ref{lma:frostman}:
$$ (\mu \times \mu) (A_\delta) = \int \mu(B(x,\delta)) \, d\mu(x) \leq C \delta^\beta
$$
so the claim follows.
\end{proof}

\subsection{Renewal operator appears}

We define another measure on $\mathrm{GL}(d,\R)$
$$\lambda = \sum_{j \in \cA} p_j \delta_{ A_j^\top},$$
which is the distribution of $A_j^\top$, the transpose of $A_j$. This is the measure we need in the renewal operator. The renewal operator adds elements on the left, while self-affinity adds elements on the right, this is the reason why we take the transpose.

In order to understand what happens in the case when $|x-y| \geq \delta$ for the sums
\begin{align}\label{eq:fourieraverages} \sum_{w \in \cW_t(z)} p_w  e^{-2\pi i A_w^\top \xi\cdot (x-y)},\end{align}
in the one-dimensional self-similar case \cite{LS}, we used the stopping time $n_t$ to write this sum as an expectation $\E_{\P_t}(g(S_t - t))$ for some suitable function $g$ depending on $\xi$. In here we need to take into account the direction $z = \xi/|\xi| \in  X$. Recall that in renewal theory (see the Section earlier on renewal theorem), we defined the following renewal operator from bounded Borel functions on $X\times\R$ to functions on $X$ by
\[\E_t f(x)=\sum_{n\geq 0}\int_{\sigma(g,x)\geq-t> \sigma(hg,x) } f(hgx,\sigma(hg,x)+t)\,d\lambda(h) \,d\lambda^{*n}(g),\]
where $X = \Sh^{d-1}$ and for an element $g \in \mathrm{GL}(d,\R)$, the map
$$\sigma(g,z) =\log\frac{|gv|}{|v|}$$ 
for $z$ in $X$ and $v\in \R^+ z = \{tz : t > 0\}$. Recall that we wrote $gz$ for a point in $X$ given by $gv/|gv|$ for $v\in\R^+z$, which gives an action of $G$ on $X$. 

 The following translation of languages allows us to analyse the averages \eqref{eq:fourieraverages} using the operators $\E_t f(x)$ as follows:

\begin{prop}\label{prop:approximation}
For all $x,y \in \R^d$, $z=\xi/|\xi|\in X$ and $t>0$ we have
$$\sum_{w \in \cW_t(z)} p_w  e^{-2\pi i A_w^\top \xi\cdot (x-y)} = \E_t g_{s_1,z_1}(z)$$
for a suitable
\begin{equation}\label{equ:gs}
g_{s_1,z_1}(z,u) = \exp(-2\pi i  s_1 \langle z,z_1\rangle e^u ),
\end{equation}
where $s_1=|\xi|e^{-t}|x-y|$ and $z_1=(x-y)/|x-y|$.
\end{prop}

\begin{proof}[Proof of Proposition \ref{prop:approximation}] Let $v$ be the unit vector in $\R^+z$ and write $g=A_w^\top $. We obtain
	\begin{align*}
		A_w^\top \xi=|\xi|gv.
	\end{align*}
	Therefore by $s_1=|\xi|e^{-t}|x-y|$,
	\begin{equation}
	A_w^\top \xi\cdot(x-y)=|x-y|\langle gv/|gv|,z_1\rangle|\xi||gv|=s_1e^{t+\sigma(g,z)}\langle gz,z_1\rangle.
	\end{equation}
	Thus after taking exponentials, by the definition of $\E_t$ and $\cW_t(z)$ we obtain the desired identity.
\end{proof}

\subsection{Proof of Theorem \ref{thm:main}} Let us now complete the proof of the main Fourier decay result, Theorem \ref{thm:main}, assuming the renewal theorem Proposition \ref{prop:stop} earlier.
\begin{proof}[The end of the proof of Theorem \ref{thm:main}]
	We first make a cutoff to be able to compute the Lipschitz norm.
Let $g_{s_1}(z,u)=g_{s_1,z_1}(z,u)\rho(u)$, where $g_{s_1,z_1}$ is defined in \eqref{equ:gs} and $\rho$ is a smooth cutoff such that $\rho_{[-|\supp \lambda|,|\supp\lambda|]}=1$ and becomes $0$ outside of $[-|\supp\lambda|-1,|\supp\lambda|+1]$. Since the operator $\E_t$ only concerns the value of the function on $X\times [-|\supp\lambda|,|\supp\lambda|]$, we obtain
\[\E_tg_{s_1}=\E_tg_{s_1,z_1}. \]

By Proposition \ref{prop:stop}, we have 
\begin{align*}
 \E_tg_{s_1}(z)  =&\frac{1}{|\sigma_{\lambda}|}\int_{X}\int_G\int^{-\sigma(h,y)}_{0} g_{s_1}(hy,\sigma(h,y)+u)\,d u\,d\lambda(h)\,d\nu_z(y)\\
&+o_t\| g_{s_1}\|_\lf.
 \end{align*}
 The term $\|g_{s_1}\|_\lf$ is bounded by $O(s)$.
To obtain the decay from high oscillation, we will first take $y, h$ such that $\langle z,hy\rangle$ is not too small, which implies that the oscillation in $g_{s_1}$ is large. Then we integrate with respect to the Lebesgue measure $u$ to obtain the decay.

More precisely, let $D(z_1,s)$ be the subset of $G\times X$ such that for $(h,y)$ in this set we have $|\langle z_1,hy\rangle|<s^{-\eps}$. By stationarity and Lemma \ref{lma:guivarch} (Guivarc'h regularity), we have 
$$\lambda\times \nu_z(D(z_1,s))=\nu_z(x\in X|\, d(x,Y_{z_1})\leq s^{-\eps})\leq Cs^{-\eps\alpha},$$
where $Y_{z_1}=\{x\in X| \langle z_1,x\rangle=0 \}$.
For $(h,y)$ not in $D(z_1,s)$, the frequency of the oscillation function $g_{s_1}$ is large, that is 
\[s_1\langle z_1,hy \rangle=|\xi|e^{-t}|x-y|\langle z_1,hy\rangle \geq s\times\delta\times s^{-\eps}=s^{1-2\eps}\geq s^{1/2}. \]
Since $\lambda$ is compactly supported, the norm cocycle $\sigma(h,y)$ is bounded. Then by an elementary estimate \cite[Lemma 3.7]{Li1}, we have
\[\int^{-\sigma(h,y)}_{0} g_{s_1}(hy,u)\,d u=O(s^{-1/2}).  \]

Therefore, we obtain a upper bound of the main part
\begin{equation}\label{eq:mainpart}
 \Big|\int_{X}\int_G\int^{-\sigma(h,y)}_{0} g_{s_1}(hy,\sigma(h,y)+u)\,d u\,d\lambda(h)\,d\nu_z(y)\Big|=O(s^{-\eps\alpha})+O(s^{-1/2}).
\end{equation}
The error term is bounded by $o_tO(s)$. Recall that $|\xi|=se^t$ and $o_t$ tends to zero as $t$ tends to infinity. When $|\xi|$ tends to infinity, with a suitable chose of $t$ and $s=o_t^{-1/2}$, then $\E_tg_{s_1,z_1}$ tends to zero. Hence by Lemma \ref{lma:cauchy}, \ref{lma:expsum} and Proposition \ref{prop:approximation}, the proof is complete.
\end{proof}
\begin{remark}
	Here the group $\Gamma_\lambda$ equals $\Gamma^\top$, the transpose of the group generated by the linear part of the affine contraction. The transpose group $\Gamma^\top$ is also strongly irreducible and proximal, which enables us to apply renewal theorem (Proposition \ref{prop:stop}). 
	
	The Zariski closure of $\Gamma^\top$ is the transpose of the Zariski closure of $\Gamma$. The transpose does not change the splitness and the connectedness. In fact, these two algebraic groups are isometric. This ensures that we can use Proposition \ref{prop:stopexp} in the following proof.
\end{remark}
\subsection{Proof of Theorem \ref{thm:mainquant}} 
We keep the notation in the proof of Theorem \ref{thm:main}. We only need to give a better bound of the error term.
By Proposition \ref{prop:stopexp} we have
	\begin{align*}
	\E_t g_{s_1}(z) = \,\, &\frac{1}{|\sigma_{\lambda}|}\int_{X}\int_G\int^{-\sigma(h,y)}_{0} g_{s_1}(hy,\sigma(h,y)+u)\,d u\,d\lambda(h)\,d\nu_z(y)\\
	&+e^{-\eps_1 t}O(e^{4\eps_1|\supp g_{s_1}|}(\|g_{s_1}\|_{\mathrm{Lip}}+\|\partial_{uu}g_{s_1}\|_{\mathrm{Lip}})).
	\end{align*}
Let us now plug-in $t$ to the exponential error term
$$e^{-\eps_1 t}O(e^{4\eps_1|\supp g_{s_1}|}(\|g_{s_1}\|_{\mathrm{Lip}}+\|\partial_{uu}g_{s_1}\|_{\mathrm{Lip}}))$$
 from $s=|\xi|e^{-t}$, which gives
	$$e^{-\eps_1 t} = \Big(\frac{s}{|\xi|}\Big)^{\eps_1} .$$
	For the Lipschitz norm, by $s_1=s|x-y|$ we obtain
	\begin{equation}\label{equ:lip}
	\|g_{s_1}\|_{Lip}=O(s), \text{ and }\|\partial_{uu}g\|_{Lip}=O(s^3).
	\end{equation}
	Setting now $s = s(\xi)$ such that 
	$$s = s(\xi) := |\xi|^{\frac{\eps_1}{6+\eps_1}}$$ 
	gives that
	\begin{equation}\label{equ:epst}
	e^{-\eps_1 t}= \Big(\frac{s}{|\xi|}\Big)^{\eps_1} = |\xi|^{-\frac{6\eps_1}{6+\eps_1}}= s^{-6}.
	\end{equation}
	Therefore, combining \eqref{equ:lip} and \eqref{equ:epst}
\begin{equation}
	e^{-\eps_1 t}O(e^{4\eps_1|\supp g_{s_1}|}(\|g_{s_1}\|_{\mathrm{Lip}}+\|\partial_{uu}g_{s_1}\|_{\mathrm{Lip}}))\leq O(|\xi|^{-\frac{3\eps_1}{6+\eps_1}}).
\end{equation}
	
Moreover, by \eqref{eq:mainpart} we then have, as $|\xi| \to \infty$, that
$$\Big|\int_{X^2}\int_G\int^{-\sigma(h,y)}_{0} g(hy,\sigma(h,y)+u)\,d u\,d\lambda(h)\,d\nu_z(y)\Big|=O(s^{-\eps_1\alpha})+O(s^{-1/2}) = O(|\xi|^{-\alpha_1}),$$
with $\alpha_1>0$.
Thus the decay rate of $|\widehat{\mu}(\xi)|$ as $|\xi| \to \infty$ is polynomial. \qed
\begin{remark}\label{rem:s2pf}
	When the self-affine set $F$ satisfies strong separation condition, then the $\beta=s_2$ in Lemma \ref{lma:expsum} is independent of the translation part. Hence the decay rate in the proof above is also independent of the translation part.
\end{remark}

\section{Proof of the renewal theorems}
\label{sec:renewal}
We start to prove the renewal theorems for random walks on the sphere $\Sh^{d-1}$, $d\geq 2$. Recall $X:=\Sh^{d-1}$ and $V:=\R^d$ equipped with a norm.
Recall that our random walk given by the measure $\lambda$ on $\mathrm{GL}(V)$, and the group $\Gamma_\lambda$ acts proximally and strongly irreducibly on $V$.
 We relax the assumption on the support to finite exponential moment, that is there exists $\eps>0$ such that
\[\int \|g\|^\eps\,d\lambda(g)< +\infty. \]
We only suppose the negativeness of the first Lyapunov exponent, $\sigma_\lambda<0$, instead of $\|g_j\|<1$.

 The first step, we will follow \cite{Boyer} to obtain the classic renewal theorem with an error term depending on some operator. The section 4 of \cite{Boyer} is written for general groups, which also works in our cases. What we borrow from the work of Boyer \cite{Boyer} is some classical estimates and functional analysis. Similar results and computations can also be find in \cite{Babillot}, \cite{Sarig} and \cite{Avila}.
 Please see the discussion after Proposition \ref{prop:1-pz}, where we give precise citations from \cite{Boyer}.
 Next, with the additional assumption that the Zariski closure is $\R$-split, we use the spectral gap in \cite{Li2} to obtain an exponentially error term in our renewal theorem. At last, we follow \cite{Li1} to obtain the renewal theorem for residue process with an error term. 

For the renewal theorem on projective spaces please see \cite{Li1}. Here we deal with renewal theorem on spheres, for more details please see \cite{Boyer} and \cite{GLP}. 
\subsection{Renewal theorem for random walks on spheres}
We follow \cite[Section 4]{Boyer} in this part.
Recall the measure $\nu_x$ defined in Definition \ref{defi:nux}.
For $x$ in $ X$ and a continuous function $f$ on $ X$, we let 
\[ N_0f(x)=\int f\,d\nu_x. \]
For a continuous function $f$ on $ X\times\R$, we define 
\[N_0f(x,t):=N_0f_t(x), \]
where $f_t(x)$ is seen as a function on $ X$.

Recall that for $g$ in $\mathrm{GL}(V)$ and $x=v$ in $ X$, we write $\sigma(g,x)=\log\frac{|gv|}{|v|}$.
Let $z$ be  a complex number and let $P_z$ be the complex transfer operator: For $\Re z$ small enough and $f$ a continuous function, $x$ in $ X$
\[P_zf(x)=\int e^{z\sigma(g,x)}f(gx)\,d\lambda(g). \]
Then the operator $N_0$ projects the function on $ X$ to the subspace of $P_0$-invariant functions. (See \cite[Lemma 2.13]{Boyer})

We define the renewal operator. For a bounded positive Borel function $f$ on $ X\times \R$, let
\[Rf(x,t)=\sum_{n\geq 0}\int f(gx,t+\sigma(g,x))\,d\lambda^{*n}(g). \]
Let $f$ be a positive bounded continuous function in $L^1( X\times\R,\nu\otimes Leb)$. We define the operator $\Pi_0$ by
\[\Pi_0f(x,t)=\int_{-\infty}^t N_0f(x,u)\,d u. \] 
Let $C^\gamma( X)$ be the space of $\gamma$-H\"older continuous functions on $ X$ and the norm is given by 
\[\|f\|_{C^\gamma}:=\|f\|_\infty+c_\gamma(f), \text{ where }c_\gamma(f):=\sup_{x\neq y}\frac{|f(x)-f(y)|}{d(x,y)^\gamma}. \]
Define the $L^\infty C^\gamma$ norm on $ X\times\R$ by 
\[\|f\|_{L^\infty C^\gamma}:=\sup_{\xi\in\R}\|f(x,\xi)\|_{C^\gamma}, \]
which is the supremum of the H\"older norm on $f(\cdot,\xi)$. Define another Sobolev norm
$$\|f\|_{W^{1,\infty}C^\gamma}:=\|f\|_{L^\infty C^\gamma}+\|\partial_\xi f\|_{L^\infty C^\gamma}.$$
Write the Fourier transform $\hat{f}(x,\xi)=\int e^{i u\xi}f(x,u)\,d u$ for $f$ in $L^1(X\times \R,\nu\otimes Leb)$.
	\begin{prop}\label{prop:1-pz}
		Let $\lambda$ be a Borel probability measure on $\mathrm{GL}(V)$ with an exponential moment, such that the group $\Gamma_\lambda$ acts proximally and strongly irreducibly on $V$. Suppose that the first Lyapunov exponent $\sigma_\lambda$ is negative. Then 
		\begin{itemize}
			\item[(i)] There exists $\gamma>0$ such that $P_z$ preserves the H\"older space $C^\gamma( X)$ when $\Re z$ small.  There exists an analytic operator $U(z)$ on $C^\gamma( X)$, defined on a neighbourhood of the imaginary line, such that for $z$ in this domain
			\[\frac{1}{Id-P_z}=\frac{N_0}{\sigma_{\lambda} z}+U(z). \]
			\item[(ii)]  Suppose that $f$ is in $L^1( X\times\R,\nu\otimes Leb)\cap W^{1,\infty}C^\gamma( X\times\R)$ and the projection of $\supp\hat f$ onto $\R$ is compact. For all $x$ in $ X$ and real number $t>0$,
			\[Rf(x,t)=\frac{1}{|\sigma_\lambda|}\Pi_0f(x,t)+\int e^{it\xi}U(i\xi)\hat{f}(x,\xi)\,d\xi. \]
		\end{itemize}	
	\end{prop}

We want to explain how to establish this renewal type theorem by using \cite[Theorem 4.1]{Boyer}. 
\begin{itemize}
	\item[(1)] Since $\Gamma_\lambda$ is strongly irreducible and proximal and $\lambda$ has finite exponential moment, by \cite[Proposition 2.3, Chapter V]{Bougerol-Lacroix}, the action of the group $\mathrm{GL}(V)$ on $C^\gamma(\P V)$ is $(\lambda,\gamma)$-contracting. This verifies the contracting condition in \cite[Theorem 4.1]{Boyer}.
	
	\item[(2)] In \cite[Theorem 4.1]{Boyer}, Boyer made another assumption on the norm of the operator $\|(Id-P_{ib})^{-1}\|_{C^\gamma( X)}$. But this condition is only used to get a larger definition region of the analytic operator $U(z)$, which is then used to obtain the renewal theorem for some regular functions. We will prove a stronger condition in the next subsection.
	
	\item[(3)] 	In our situation, we only need the existence of $U(z)$ in a neighbourhood of the imaginary line. This is due to the fact that $P_{ib}$ for $b\in\R$ has $1$ as eigenvalue only at $b=0$, which is a consequence of the non-arithmeticity of the cocycle $\sigma(g,x)$. Under our proximal and strongly irreducible assumption, the non-arithmeticity can be found in \cite[Proposition 2.5]{GLP} or \cite[Theorem 7.4]{benoistquint}.
	
	\item[(4)] 	In the one dimensional case \cite{LS}, we need the extra assumption that the contracting ratio is non-arithmetic. But in higher dimension, the non-arithmeticity is automatically given by the proximal and strongly irreducible condition.
	
	\item[(5)] 	Hence \cite[Lemma 4.9]{Boyer} gives (i). And the renewal theorem is for the function whose Fourier transform has a compact support. By the same computation in \cite[Proposition 4.14]{Boyer} or \cite[Proposition 4.5]{Li1} (the same as for the renewal theorem on $\R$), we can establish a renewal theorem with an error term. 
\end{itemize}
\begin{remark}
	In our renewal theorem, the main term is given by the integration on $(-\infty,t)$, because the Lyapunov constant $\sigma_\lambda$ is negative and the the cocycle $\sigma(g,x)$, with $g$ following the law of $\lambda^{*n}$, tends to $-\infty$ as $n$ tends to $+\infty$.
\end{remark}

\subsection{Exponential error term}
	Now, we want to explain how to obtain the exponential error term in renewal theorem under the additional assumption that the Zariski closure of $\Gamma_\lambda$ is connected and $\R$-split. Recall that a Borel probability measure $\lambda$ on an algebraic group $G$ is called Zariski dense if the group $\Gamma_\lambda$ is Zariski dense in $G$. The key input is the following uniform spectral gap for complex transfer operator
	\begin{lemma}[Spectral gap]\label{prop:spectralgap}
		Let $\lambda$ be a Zariski dense Borel probability measure on a connected reductive group $G$ defined and split over $\R$ with a finite exponential moment. For $\gamma>0$ small enough, there exist $\rho<1, C>0$ such that for all $a$, $b$ in $\R$ with $|b|$ large enough, $|a|$ small enough and $f$ in $C^{\gamma}( X)$, $n$ in $\N$ we have
		\begin{equation*}
		\|P^{n}_{a+ib}f\|_{\gamma}\leq C|b|^{2\gamma}\rho^n\|f\|_{\gamma}.
		\end{equation*}
	\end{lemma}
	This spectral gap is established in \cite{Li2} Theorem 4.19 for semisimple groups and transfer operators on the projective spaces. We will indicate the modification needed to prove this version later.
	
	The spectral gap of $P_z$ implies that the analytic operator $U(z)$ in Proposition \ref{prop:1-pz} has an analytic continuation to a strip of the imaginary line and the operator norm of $U(z)$ is bounded by a polynomial of the imaginary part.
	We can obtain an exponential error term in renewal theorem by the same approach as in \cite{Li2} Section 4.3. The strengthened version of renewal theorem is as follows
	\begin{prop}\label{prop:renerr}
		Let $G$ a connected reductive group defined and split over $\R$, which acts irreducibly on $V$.
			Let $\lambda$ be a Zariski dense Borel probability measure on $G$ with a finite exponential moment and $\sigma_\lambda<0$. For every $\gamma>0$ small enough, there exists $\eps>0$ such that for $f$ in $C_c^\infty( X\times\R)$, all $x$ in $ X$ and $t$, we have
			\[Rf(x,t)=\frac{1}{|\sigma_\lambda|}\Pi_0f(x,t)+e^{-\eps|t|}O(e^{\eps|\supp f|}(\|\partial_{tt}f\|_{L^1_\R C^\gamma_{ X}}+\|f\|_{L^1_\R C^\gamma_{ X}})), \]
			where $|\supp f|=\sup\{|u|,\ (x,u)\in\supp f \}$ and 
			$$\|f\|_{L^1_\R C^\gamma_{ X}}=\int \|f(x,t)\|_{C^\gamma({ X})}\,d t.$$
	\end{prop}
	
	It remains to prove the spectral gap (Lemma \ref{prop:spectralgap} )
	
\begin{proof}[Proof of spectral gap (Lemma \ref{prop:spectralgap})]
	In \cite{Li2}, we deduce Theorem 4.19 (analogue of Lemma \ref{prop:spectralgap}) from a priori estimate (Proposition 4.22, analogue of Lemma \ref{prop:spectral real}) and a $L^2$ estimate (Proposition 4.23, analogue of Lemma \ref{prop:l2}). We need to establish the analogue of Proposition 4.22 in \cite{Li2}.
	\begin{lemma}\label{prop:spectral real}
		With the same assumption as in Lemma \ref{prop:spectralgap}, for every $\gamma>0$ small enough, there exist $C>0$ and $0<\rho<1$ such that for all $f$ in $C^\gamma({ X})$, $|a|$ small enough and $n\in\N$
		\begin{align}
		&\label{equ:expmom}
		\|P^n_zf\|_\infty\leq C^{|a|n}\|f\|_\infty,\\
		&\label{equ:infpnf}\|P_0^nf\|_{\infty}\leq\left\|N_0f\right\|_\infty+C\rho^n\|f\|_{\gamma},\\
		&\label{equ:gampnf}
		c_\gamma(P^n_zf)\leq C(C^{|a|n}(1+|b|^{\gamma})\|f\|_{\infty}+\rho^nc_\gamma(f)).
		\end{align}
	\end{lemma}
	The first inequality is due to exponential moment and the Cauchy-Schwarz inequality. 
	
	The third inequality \eqref{equ:gampnf} can be proved similarly as in Lemma 4.6 of \cite{Boyer}. Boyer only proved similar inequality for $\Re z\geq 0$, but the same argument also works for all $|\Re z|$ small. The term $|b|^\gamma$ is $|z|$ in \cite{Boyer}, which comes from \cite[Page 57, Line 7]{Boyer}. But this term can be replaced by $|z|^\gamma$ if we replace that inequality by a sharper inequality \cite[Page 133, line 14]{Li-thesis}. 
	
	By \eqref{equ:gampnf} with $z=0$ and Theorem (Ionescu-Tulcea and Marinescu in [ITM50]), we know that $P_0$ has essential spectral radius $r$ less than $1$. This means that in the subset of the complex plane $B(0,1)-B(0,r)$, the spectral values of operator $P_0$ are eigenvalues and countable. Moreover, the possible accumulation points of the eigenvalues are in $B(0,r)$. Therefore, since $P_0$ restricted to $\ker N_0$ has no eigenvalue of absolute value $1$ ($N_0$ is a projection on $C^\gamma(\Sh^{n-1})$ whose image is exactly the $P_0$-invariant subspace of $C^\gamma(\Sh^{n-1})$), we know that $P_0$ has spectral radius less than $1$ in $\ker N_0$. We conclude that there exist $\rho<1$ and $C>0$ such that for $n\in\N$ and $f\in C^\gamma({ X})$,
	\[ \|(P_0^n-N_0)f\|_\gamma=\|P_0^n(1-N_0)f\|_\gamma\leq C\rho^n\|f\|_\gamma. \]
	In particularly, this implies \eqref{equ:infpnf}.
	
	Hence by the same argument as in \cite[page 134]{Li-thesis}, we only need to establish a similar $L^2$ estimate (Proposition 4.23 in \cite{Li2}) in our case. The proof of spectral gap (Lemma \ref{prop:spectralgap}) is complete by the following Lemma.
	\end{proof}
	It is useful to take a different regularity norm, for $f$ in $C^\gamma( X)$ and $b$ in $\R^* = \R \setminus \{0\}$, let 
		\[ \|f\|_{\gamma,b} :=\|f\|_\infty+c_\gamma(f)/|b|^\gamma. \]	
	\begin{lemma}\label{prop:l2}
		With the same assumption as in Lemma \ref{prop:spectralgap}, for every $\gamma>0$ small enough, for $|b|$ large enough and $|a|$ small enough, there exist $\eps_2,C_2>0$ such that for $f$ in $C^\gamma$ with $\|f\|_{\gamma,b}\leq 1$ and any $\lambda$-stationary measure $\nu$ on $ X$, we have
		\[\int|P^{[C_2\ln|b|]}_{a+ib}f|^2\,d\nu\leq e^{-\eps_2\ln|b|}. \]
	\end{lemma} 
	\begin{proof}
		Recall that Proposition 4.23 in \cite{Li2} is the same statement for connected semisimple algebraic groups defined and split over $\R$ and the stationary measure is on projective spaces or flag varieties. We will indicate the modification needed to prove our case. In the proof of Proposition 4.23, using Cauchy-Schwarz inequality, then we apply the Fourier decay of stationary measure to the following quantity for $g,h$ in $G$
		\[A_{g,h}:=\int_{ X} e^{z\sigma(g,w)+\bar{z}\sigma(h,w)}f(gw)\bar{f}(hw)\,d\nu(w). \]
		
		\textbf{Step 1}: Reduce the spherical case to projective case. We need to separate the functions on the sphere. For $w\in  X$, let $r(w)=-w$ be its antipodal point in $ X$. Then every function $f$ on $ X$ can be separated into 
		\[f=f_1+f_2 \text{ with } f_1(w)=\frac{1}{2}(f(w)+f(r(w))),\ f_2(w)=\frac{1}{2}(f(w)-f(r(w))). \]
		Then $f_1$ is invariant under the action of $r$ and $f_2$ goes to its additive inverse under the action of $r$. What's more, the $C^\gamma$ norm of $f_1,f_2$ is less than $f$. Let $C^\gamma_1$ and $C^\gamma_2$ be the spaces of $\gamma$ H\"older function on $ X$ which is invariant and anti-invariant under the action of $r$, respectively.
		
		We will prove Lemma \ref{prop:l2} for $C^\gamma_1$ and $C^\gamma_2$, then Lemma \ref{prop:l2} also holds for $C^\gamma$. 
		For $f\in C^\gamma_j$ with $j=1$ or $2$, we know that the product
		\[f_{g,h}(w)=f(gw)\bar{f}(hw) \]
		is in $C^\gamma_1$. Let $u=\pi(w)$, where $\pi$ is the map from $ X$ to $\P V$. The cocycle function is actually defined on $\P V$, which is also invariant under $r$. Therefore
		\[A_{g,h}=\int e^{z\sigma(g,u)+\bar{z}\sigma(h,u)}f_{g,h}(u)\,d\nu_{\P V}(u), \]
		where $\nu_{\P V}$ is the unique $\lambda$-stationary measure on $\P V$ which is also the pushforward of the measure $\nu$ under the map $\pi$.
		Then we can continue as in the proof of Proposition 4.23 for projective spaces. 
		
		\textbf{Step 2}: Reduce the reductive case to the semisimple case. There is no essential difficulty and the following is a conceptual argument. (This is the only part we really need theory of algebraic groups. For the first time reading, you can assume $G=\mathrm{GL(2,\R)}$. For more details of algebraic groups, please see \cite{borel} and \cite{benoistquint})
		
		We recall some notation from \cite[Section 2.1]{Li2},  the Lie algebra $\mathfrak a$ is the Lie algebra of a maximal torus $A$ of $G$, the semisimple part of $\mathfrak a$ is $\mathfrak b$, the flag variety $\cP$, the Iwasawa cocycle $\sigma_I$ from $G\times \cP$ to $\mathfrak a$.
		The norm cocycle can be separated into central part and semisimple part for $g$ in $G$ and $u$ in $\P V$ (See for example \cite[Page 8]{Li2})
		\[\sigma(g,u)=c(g)+\sigma_{ss}(g,u).\]
		For example when the Zariski closure is $\mathrm{GL}(V)$, then the central part is the logarithm of the absolute value of the determinant. For more details of this example, please see the example in \cite[Page 11]{Li2}.
		We can write
		\[A_{g,h}=e^{zc(g)+\bar{z}c(h)}\int e^{z\sigma_{ss}(g,u)+\bar{z}\sigma_{ss}(h,u)}f_{g,h}(u)\,d\nu_{\P V}(u). \]
		The integral only involves the semisimple part and we want to use the Fourier decay on the flag variety to deal with it.
		
		Let $\pi_V$ be the map from the flag variety $\cP$ to the projective space $\P V$. 
		Since the map $\pi_V$ is $G$-equivalent, we can lift every thing from $\P V$ to the flag variety  $\cP$, the $\lambda$-stationary measure $\nu_{\P V}$ to the $\lambda$-stationary measure $\nu_{\cP}$, the function $f$ on $\P V$ to a function $\tilde{f}$ on $\cP$, the cocycle $\sigma_{ss}(g,u)$ to $\sigma_{ss}(g,\eta_u)$, where $\eta_u$ is in $\cP$ such that $\pi_V(\eta_u)=u$.
		
		Let $q$ be the quotient map from $G$ to its quotient by the connected centre $C$, that is $q: G\rightarrow G_1=G/C$.
		The quotient group $G_1$ is a connected semisimple algebraic group defined and split over $\R$, which satisfies the hypothesis of Proposition 4.23 in \cite{Li2}. Then $q_*(\lambda)$ is also a Zariski dense measure on $G_1$. 
			\begin{center}
							\begin{tikzcd}
								G \arrow[d, "q"']  \arrow[r, ""] & \bf \cP \arrow [d, "\pi_V "]  \\
								G_1 \arrow[ru,dotted]  & \P V
							\end{tikzcd}
			\end{center}	
		Since the action of $G$ on $\cP$ factors through $G_1=G/C$, we have $g\eta=q(g)\eta$ for $g\in G$ and $\eta\in\cP$ and the $q_*(\lambda)$-stationary measure on $\cP$ is also $\nu_{\cP}$. Since $q$ induces an injective map on $\mathfrak b$, the Lie algebra $\mathfrak b$ can also be seen as the Lie algebra of a maximal torus in $G_1$. For the semisimple part of the cocycle, there exists a nonzero weight $\chi$ in the dual space of the Lie algebra $\mathfrak b$ such that
		\begin{equation}\label{equ:sigss}
		\sigma_{ss}(g,\eta)=\chi\sigma_I(q(g),\eta), 
		\end{equation}
		where $\sigma_I(q(g),\eta)$ is the Iwasawa cocycle on $G_1\times\cP$ which takes value in the Lie algebra $\mathfrak b$. Then as in the proof of Proposition 4.23, we can apply the Fourier decay of the stationary measure on flag variety for $q_*(\lambda)$ on $G_1$ and we conclude that the integral
		\[\int e^{z\sigma_{ss}(g,u)+\bar{z}\sigma_{ss}(h,u)}f_{g,h}(u)\,d\nu_{\P V}(u)=\int e^{z\chi\sigma_I(q(g),\eta)+\bar{z}\chi\sigma_I(q(h),\eta)}\tilde f_{g,h}(\eta)\,d\nu_{\cP}(\eta) \]
		 is small for most pairs $(g,h)\in G\times G$. 
		
		For the extra term, we know 
		\[|e^{zc(g)+\bar{z}c(h)}|= e^{\Re z(c(g)+c(h))}. \]
		We will only sum up $g,h$ with the law of $\lambda^{*n}$, where $n=[C_2\ln|\theta|]$. If we assume that the support of $\lambda$ is compact, then we know $|c(g)|,|c(h)|\leq Cn$. When $|\Re z|$ is small enough, this term is less than
		\[e^{|\Re z|Cn}\leq e^{\eps_2\ln|\theta|/2}. \]
		(For finite exponential moment case, by using Large deviation principle, we can obtain similar result.) 

	This completes the proof of Lemma \ref{prop:l2}.
\end{proof}
\begin{remark}
	One technical point is that in the above proof we need the norm on $V$ is ``good" with respect to $G$, which enables us to compare the norm cocycle and the Iwasawa cocycle in \eqref{equ:sigss}. Lemma \ref{prop:spectralgap} and Proposition \ref{prop:renerr} actually hold for any norm on $V$, see \cite[Remark 4.21]{Li2}.
\end{remark}

\subsection{Renewal theorem for Residue process}

We consider the residue process for the cutoff of a function $f$ on $X\times\R^2$.

\begin{definition}[Residue process]
	Define the operator $\Car$ from bounded Borel functions on $X\times\R^2$ to functions on $X\times\R$ by
	\[\Car f(x,t)=\sum_{n\geq 0}\int_{\sigma(g,x)\geq -t> \sigma(hg,x) } f(hgx,\sigma(h,gx),\sigma(g,x)+t)\,d\lambda(h) \, \,d\lambda^{ \ast n}(g). \]
	
\end{definition}

To state the renewal theorem, that is, the asymptotics of $\Car f(x,t)$, we need to talk about the Lipschitz regularity of the test functions $f : X \times \R^2 \to \C$. Define the \textit{Lipschitz norm} of $f$ by  
	\[\|f\|_{\mathrm{Lip}}=\|f\|_\infty+\sup_{(x,v,u)\neq(x',v',u')}\frac{|f(x,v,u)-f(x',,v',u')|}{d(x,x')+|v-v'|+|u-u'|}.\]
Note that this extends the definition of the Lipschitz norm functions on $X \times \R$ we used in Proposition \ref{prop:stop} and Proposition \ref{prop:stopexp} earlier.


\begin{prop}[Renewal theorem irreducibility and proximality]\label{prop:rescar} 
	Let $\lambda$ be a Borel probability measure on $\mathrm{GL}(V)$ with an exponential moment, such that the group $\Gamma_\lambda$ acts proximally and strongly irreducibly on $V$. Suppose that the first Lyapunov exponent $\sigma_\lambda$ is negative.
	
	Let $ f$ be a continuous function on $X\times\R^2$ with $\| f\|_\lf$ finite. Assume that the projection of $\supp f$ on $\R_v$ is contained in a compact set $K$. For all $\delta>0$, $t>\max\{2(|K|+\delta),20\}$ and $x$ in $X$, we have, as $t \to +\infty$ that
	\begin{align*}
	\Car f(x,t) = \,\, & \frac{1}{|\sigma_{\lambda}|}\int_{X}\int_G\int^{-\sigma(h,y)}_{0} f(hy,\sigma(h,y),u)\,\,d u\,\,d\lambda(h)\,\,d\nu_x(y)\\&+O_K(\delta+O_\delta/t)\| f\|_\lf,
	\end{align*}
	where $O_K$ does not depend on $\delta, f,t,x$, the integral $\int^{-\sigma(h,y)}_{0}=0$ if $\sigma(h,y)>0$.
\end{prop}

By the same argument as in \cite{Li1}, we can establish the renewal theorem for residue process from the classic renewal theorem, that is Proposition 4.17 in \cite{Li1}. This gives a proof of our renewal theorem Proposition \ref{prop:rescar}.

For the version with exponential error term, the argument is similar. We will establish analogous versions of Proposition 4.15  and Proposition 4.17 in \cite{Li1}. Since we have a very strong error term in Proposition \ref{prop:renerr}, the argument will be much more direct. 
Here we need to consider higher order regularity, similar to assuming $f$ is a Sobolev function, and we define the following $L^1$-\textit{Lipschitz norm} of $f : X \times \R^2 \to \C$ by
\begin{align}\|f\|_{L^1\mathrm{Lip}} :=\int\left(\sup_{x,v}|f(x,v,u)|+\sup_{(x,v)\neq (x',v')\in X\times\R}\frac{|f(x,v,u)-f(x',v',u)|}{d(x,x')+|v-v'|} \right)\,\,d u \label{eq:l1lip}. \end{align}

\begin{prop}[Renewal theorem for $\R$-splitting Lie groups on the sphere]\label{prop:rescarexp}
	Let $G$ a connected reductive group defined and split over $\R$, which acts irreducibly on $V$.
	Let $\mu$ be a Zariski dense Borel probability measure on $G$ with a finite exponential moment and $\sigma_\mu<0$. There exists $\eps>0$ such that the following holds. Let $f$ be a smooth compactly supported function on $ X\times \R^2$. Then for $t>0$ and $x\in X$, we have
	\begin{align*}
	\Car f(x,t) = \,\, &\frac{1}{|\sigma_{\lambda}|}\int_{X}\int_G\int^{-\sigma(h,y)}_{0} f(hy,\sigma(h,y),u)\,d u\,d\lambda(h)\,d\nu_x(y)\\
	&+e^{-\eps t/4}O(e^{\eps|\supp f|}(\|f\|_{L^1\mathrm{Lip}}+\|\partial_{uu}f\|_{L^1\mathrm{Lip}})),
	\end{align*}
	where the integral $\int^{-\sigma(h,y)}_{0}=0$ if $\sigma(h,y)>0$.
\end{prop}

Let $f$ be a positive bounded Borel function on $ X\times\R^2$. For $(x,t)\in  X\times\R$, we define the residue operator by
\[ Ef(x,t)=\sum_{n\geq 0}\int f(hgx,\sigma(h,gx),\sigma(g,x)+t)\,d\lambda^{*n}(g)\,d\lambda(h). \]
\begin{lemma}\label{prop:residue}
	With the same assumption as in Proposition \ref{prop:rescarexp}, let $f$ be a smooth compactly supported function on $ X\times \R^2$. Then for $t>0$ and $x\in X$, we have
	\begin{equation*}
	\begin{split}
	E f(x,t)&=\frac{1}{|\sigma_{\lambda}|}\int_{X}\int_G \int_{-\infty}^{t} f(hy,\sigma(h,y),u)\,d u\,d\lambda(h)\,d\nu_x(y)\\
	&+e^{-\eps t}O(e^{\eps|\supp f|}(\|f\|_{L^1\mathrm{Lip}}+\|\partial_{uu}f\|_{L^1\mathrm{Lip}} )).
	\end{split}
	\end{equation*}
\end{lemma}
\begin{proof}
	By the same proof as in Proposition 4.15 in \cite{Li1}, we only need to compute the $L^1_\R C^\gamma_{ X}$ norm of $Qf$ and $\partial_{tt}Qf$, where 
	\[Qf(x,t)=\int_G f(hx,\sigma(h,x),t)\,d\lambda(h). \]
	Since the measure $\lambda$ has exponential moment, when $\gamma$ is small enough, by using Lipschitz property of the distance and the norm cocycle, there exists $C_\gamma>0$ such that
	\[\|Qf\|_{L^1_\R C^\gamma_{ X}}\leq C_\gamma \|f\|_{L^1\mathrm{Lip}}. \]
	For the main term, we see
	\[\Pi_0 Qf(x,t)=\int_{-\infty}^t\int_X Qf(y,u)\,d\nu_x(y)\,d u=\int_{X}\int_G \int_{-\infty}^{t} f(hy,\sigma(h,y),u)\,d u\,d\lambda(h)\,d\nu_x(y). \]
	The proof is complete.
\end{proof}
Now we can prove Proposition \ref{prop:rescarexp}.
\begin{proof}[Proof of Proposition \ref{prop:rescarexp}]
	We take a smooth cutoff $\phi$ such that $\phi |_{[0,\infty)}=1$, $\supp\phi\subset[-1,\infty)$ and $\phi$ takes value in $[0,1]$. For $\delta>0$, let $\phi_\delta(x)=\phi(x/\delta)$. Let
	\[f_\delta(x,v,u)=\phi_\delta(u)\phi_\delta(-v-u)f(x,v,u). \]
	Then

	\begin{equation}\label{equ:Efdelta}
	|Ef_\delta(x,t)-\Car f(x,t)|\leq |f|_\infty E(\BB_{-\delta\leq u\leq 0}\BB_{0\leq v+u\leq \delta})(x,t)\leq |f|_\infty R(\BB_{-\delta\leq u\leq 0})(x,t),
	\end{equation}
	where $\BB$ is the indicator function
	We take a smooth function $\psi(x,u)=\varphi_\delta(u+\delta)\varphi_\delta(-u)$ to bound the function $\BB_{-\delta\leq u\leq 0}$. Then by Proposition \ref{prop:renerr} and \eqref{equ:Efdelta}, we obtain
	\begin{equation}\label{equ:Efdel1}
	|Ef_\delta(x,t)-\Car f(x,t)|\leq R(\psi)(x,t)\leq C(\delta+e^{-\eps t}(1+\delta^{-1})),
	\end{equation}
	where $C>0$ only depends on $\lambda$ and $\phi$.
	By Lemma \ref{prop:residue}, we see
	\begin{equation}\label{equ:Efdel2}
	\begin{split}
		Ef_\delta(x,t)&=\frac{1}{|\sigma_{\lambda}|}\int_{X}\int_G \int_{-\infty}^{t} f_\delta(hy,\sigma(h,y),u)\,d u\,d\lambda(h)\,d\nu_x(y)\\
		&+e^{-\eps t}O(e^{\eps|\supp f|}(\|f_\delta\|_{L^1\mathrm{Lip}}+\|\partial_{uu}f_\delta\|_{L^1\mathrm{Lip}} )).
	\end{split}
	\end{equation}
	For the major term in \eqref{equ:Efdel2},
	\begin{align}\label{equ:major}
		&\int_{X}\int_G \int_{-\infty}^{t} f_\delta(hy,\sigma(h,y),u)\,d u\,d\lambda(h)\,d\nu_x(y) \nonumber\\
		&=\int_{X}\int_G \int_{-\infty}^{t}\phi_\delta(u)\phi_\delta(-u-\sigma(h,y)) f(hy,\sigma(h,y),u)\,d u\,d\lambda(h)\,d\nu_x(y) \nonumber\\
		&=\int_{X}\int_G\int^{-\sigma(h,y)}_{0} f(hy,\sigma(h,y),u)\,d u\,d\lambda(h)\,d\nu_x(y)+O(\delta\|f\|_\infty).
	\end{align}
	For the error term in \eqref{equ:Efdel2},
	\begin{equation}\label{equ:error}
		\|f_\delta\|_{L^1\mathrm{Lip}}\leq \delta^{-1} \|f\|_{L^1\mathrm{Lip}},\ \|\partial_{uu}f\|_{L^1\mathrm{Lip}}\leq \delta^{-3}\|\partial_{uu}f\|_{L^1\mathrm{Lip}}.
	\end{equation}
	Combine \eqref{equ:Efdel1}-\eqref{equ:error} and take $\delta=e^{-\eps t/4}$. The proof is complete.
\end{proof}

\subsection{Proofs of the renewal theorems Propositions \ref{prop:stop} and \ref{prop:stopexp}}\label{sec:proofrensphere}

Having now proved the renewal theorems Proposition \ref{prop:rescar} and its quantitative form Proposition \ref{prop:rescarexp}, we can now complete the proofs of Proposition \ref{prop:stop} and Proposition \ref{prop:stopexp} from Section \ref{sec:rensphere} earlier, which we needed for the proofs of Theorem \ref{thm:main} and Theorem \ref{thm:mainquant}.

\begin{proof}[Proof of Proposition \ref{prop:stop}]
		Let $\rho$ be a smooth cutoff such that $\rho_{[-|\supp \lambda|,|\supp\lambda|]}=1$ and becomes $0$ outside of $[-|\supp\lambda|-1,|\supp\lambda|+1]$.
		Take $f_1(x,v,u)=f(x,v+u)\rho(v)\rho(u)$. Then $f_1(x,v,u)=f(x,v+u)$ when $v,u$ are in the interval $[-|\supp\lambda|,|\supp\lambda|]$. By definition and the hypothesis that $\|g\|<1$ in the support of $\lambda$, we have
		\[\E_tf(x)=\Car f_1(x,t). \]
		This function $f_1$ satisfies the conditions in Proposition \ref{prop:rescar}, and the proof is complete by using Proposition \ref{prop:rescar}.
\end{proof}

For the quantitative version Proposition \ref{prop:stopexp}, the proof is the same as Proposition \ref{prop:stop}, using Proposition \ref{prop:rescarexp} instead of Proposition \ref{prop:rescar}. 

\begin{proof}[Proof of Proposition \ref{prop:stopexp}]
	We should notice that here we only need Lipschitz norm, but the norm in Proposition \ref{prop:rescarexp} is more complicate. With the same notation as in the proof of Proposition \ref{prop:stop}, we actually have (change of norm)
	\[\|f_1\|_{L^1 \mathrm{Lip}}\leq C\|f\|_{\mathrm{Lip}}, \]
	where $C>0$ only depends on $|\supp\lambda|$ and the $L^1$ Lipschitz norm $\|\cdot \|_{L^1 \mathrm{Lip}}$ is defined in \eqref{eq:l1lip} earlier.
\end{proof}

\bibliographystyle{plain}

\end{document}